\documentclass[11pt]{amsart}
\usepackage{a4wide}
\usepackage{amsmath}
\usepackage[utf8]{inputenc}
\usepackage{amssymb}
\usepackage{amsopn}
\usepackage{epsfig}
\usepackage{amsfonts}
\usepackage{latexsym}
\usepackage{enumitem}
\usepackage{graphicx}
\usepackage{calrsfs}
\usepackage{tikz}
\usepackage{enumerate}
\usepackage{dsfont}
\usepackage{cite}
\usepackage[colorlinks,linkcolor=blue,anchorcolor=blue,citecolor=blue]{hyperref}

\newtheorem{theorem}{Theorem}[section]
\newtheorem{lemma}[theorem]{Lemma}
\newtheorem{proposition}[theorem]{Proposition}
\newtheorem{corollary}[theorem]{Corollary}
\newtheorem*{Theorem}{Theorem}
\newtheorem*{Question}{Question}

\theoremstyle{definition}

\theoremstyle{remark}

\numberwithin{equation}{section}

\newcommand{\R}{\ensuremath{\mathbb{R}}}
\newcommand{\N}{\ensuremath{\mathbb{N}}}

\newcommand{\f}{\infty}

\begin{document}

\title[The coincidence of R\'{e}nyi-Parry measures]{The coincidence of R\'{e}nyi-Parry measures for $\beta$-transformation}

\author[Y. Huang]{Yan Huang}
\address[Y. Huang]{College of Mathematics and Statistics, Chongqing University, Chongqing 401331, People's Republic of China.}
\email{yanhuangyh@126.com}

\author[Z. Wang]{Zhiqiang Wang}
\address[Z. Wang]{College of Mathematics and Statistics, Center of Mathematics, Key Laboratory of Nonlinear Analysis and its Applications (Ministry of Education), Chongqing University, Chongqing 401331, People's Republic of China.}
\email{zhiqiangwzy@163.com}

\date{\today}

\begin{abstract}
We present a complete characterization of two different non-integers with the same R\'{e}nyi-Parry measure.
We prove that for two non-integers $\beta_1 ,\beta_2 >1$, the R\'{e}nyi-Parry measures coincide if and only if $\beta_1$ is the root of equation $x^2-qx-p=0$, where $p,q\in\mathbb{N}$ with $p\leq q$, and $\beta_2 = \beta_1 + 1$, which confirms a conjecture of Bertrand-Mathis in \cite[Section III]{Bertrand-1998}.
\end{abstract}



\maketitle

\section{Introduction}\label{sec:Introduction}
For $\beta >1$, the $\beta$-transformation $T_{\beta}:[0,1) \to[0,1)$ is defined by $$T_{\beta} (x) :=\beta x \pmod 1.$$
The non-integer $\beta$-transformations were first introduced by R\'{e}nyi in 1957 \cite{Renyi-1957}.
For any non-integer $\beta>1$, it's well-known that the Lebesgue measure is no longer $T_\beta$-invariant, but R\'{e}nyi \cite{Renyi-1957} proved that there exists a unique $T_\beta$-invariant Borel probability measure $\nu_\beta$ equivalent to Lebesgue measure.
Later, Parry \cite{Parry-1960} gave an explicit formula of the density function of $\nu_\beta$. Moreover, $\nu_\beta$ is the unique measure with maximal entropy $\log \beta$ \cite{Hofbauer-1978,Takahashi-1973}. The measure $\nu_\beta$ is called the \emph{R\'{e}nyi-Parry measure} for $\beta$-transformation.
For any integer $\beta >1$, it's clear that $\nu_\beta$ is the Lebesgue measure.

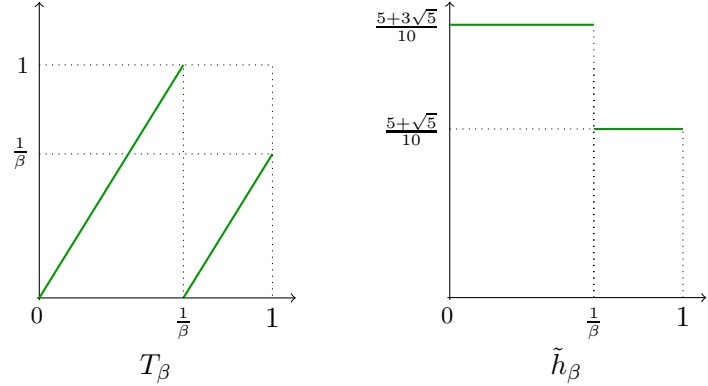
\begin{figure}[h]
\begin{tikzpicture}[scale=3.1]
\draw[->](-.01,0)node[below]{\footnotesize $0$}--(0.6180334,0)node[below]{\footnotesize $\frac{1}{\beta}$}--(1,0)node[below]{$1$}--(1.1,0);
\draw (0.5,-0.2)node[below]{$T_\beta$};
\draw[->](0,-.01)--(0,1)node[left]{\footnotesize $1$}--(0,1.27);\draw[dotted](0,0.6180334)node[left]{\footnotesize $\frac{1}{\beta}$}--(1,0.6180334);
\draw[thick, green!60!black] 
(0,0)--(0.6180334,1)(0.6180334,0)--(1,0.6180334);
\draw[dotted](0,1)--(1,1)--(1,0)(0.6180334,0)--(0.6180334,1);
\end{tikzpicture}
\quad\quad
\begin{tikzpicture}[scale=3.1]
\draw[->](-.01,0)node[below]{\footnotesize $0$}--(0.6180334,0)node[below]{\footnotesize $\frac{1}{\beta}$}--(1,0)node[below]{$1$}--(1.1,0);
\draw (0.5,-0.17)node[below]{$\tilde{h}_\beta$};
\draw[->](0,-.01)--(0,0.7236068)node[left]{\footnotesize $\frac{5+\sqrt 5}{10}$}--(0,1.1708204)node[left]{\footnotesize $\frac{5+3\sqrt 5}{10}$}--(0,1.27);
\draw[thick, green!60!black] 
(0,1.1708204)--(0.6180334,1.1708204)(0.6180334,0.7236068)--(1,0.7236068);
\draw[dotted](0.6180334,0)--(0.6180334,1)(0.6180334,0)--(0.6180334,1.1708204)(1,0)--(1,0.7236068)(0,0.7236068)--(0.6180334,0.7236068);

\end{tikzpicture}
{\caption{The $\beta$-transformation $T_{\beta}(x)$ and density function $\widetilde{h}_\beta(x)$ for $\beta=\frac{1+\sqrt 5}{2}$.}}
\label{f:1}
\end{figure}

To describe the density function of $\nu_\beta$, we define $T_\beta^0 (1) := 1$, $T_\beta^1 (1) :=\beta \pmod 1$, and inductively $T_\beta^n( 1 ):= T_\beta\big( T_\beta^{n-1} (1) \big)$ for $n \ge 2$.
Let
\begin{equation}\label{eq:nor-density}
\widetilde{h}_\beta(x) := \frac{1}{K_\beta}\sum_{ x < T_\beta^n (1)} \frac{1}{\beta^n},\quad  x \in [0,1),
\end{equation}
where the summation is over all $n \ge 0$ satisfying $T_\beta^n (1)>x$, and $K_\beta$ is the normalization constant given by
\begin{equation*}
K_\beta := \sum_{n=0}^{\f} \frac{T_\beta^n (1)}{\beta^n}.
\end{equation*}
Parry \cite{Parry-1960} showed that $\widetilde{h}_\beta(x)$ is the Radon-Nikodym derivative of $\nu_\beta$ with respect to Lebesgue measure.

As usual, a \emph{Pisot number} is an algebraic integer greater than $1$ whose algebraic conjugates are of modulus strictly less than $1$.
The \emph{degree} of a Pisot number is the degree of its minimal polynomial.
Two positive real numbers $a,b>0$ are said to be \emph{multiplicatively independent}, denoted by $a\nsim b$, if $\log a / \log b \notin \mathbb{Q}$.
Hochman and Shmerkin \cite{Hochman-Shmerkin-2015} proved the following measure rigidity result for $\beta$-trans\-for\-ma\-tions.
\begin{Theorem}\cite[Corollary 1.11]{Hochman-Shmerkin-2015}
  Let $\beta_1, \beta_2 > 1$ with $\beta_1\nsim \beta_2$ and $\beta_1$ a Pisot number.
  If $\mu$ is jointly invariant under $T_{\beta_1},T_{\beta_2}$, and if all ergodic components of $\mu$ under $T_{\beta_2}$ have positive entropy, then $\mu$ is the common R\'{e}nyi-Parry measure for $\beta_1$ and $\beta_2$; in particular, $\mu$ is absolutely continuous.
\end{Theorem}

After presenting the above result, Hochman and Shmerkin raised the following questions.
\begin{Question}
  For what pairs $(\beta_1,\beta_2)$ the R\'{e}nyi-Parry measures coincide, i.e.,  $\nu_{\beta_1} = \nu_{\beta_2}$?
\end{Question}

For two integers $\beta_1, \beta_2 > 1$, $\nu_{\beta_1} = \nu_{\beta_2}$ is the Lebesgue measure.
If $\beta_1> 1$ is an integer and $\beta_2>1$ is not, then $\nu_{\beta_1} \ne \nu_{\beta_2}$ because the Lebesgue measure is not $T_{\beta_2}$-invariant.
The situation becomes more complicated when $\beta_1, \beta_2 > 1$ are two non-integers, for which
Bertrand-Mathis has made a conjecture in \cite[Section III]{Bertrand-1998}.
The purpose of this paper is to give a complete characterization of two non-integers with the same R\'{e}nyi-Parry measure, which provides an affirmative answer to Bertrand-Mathis' conjecture.
\begin{theorem}\label{main-result}
For two different non-integers $\beta_1, \beta_2 > 1$, the R\'{e}nyi-Parry measures coincide if and only if $\beta_1$ is the root of equation $x^2-qx-p=0$, where $p,q\in\mathbb{N}$ with $p\leq q$, and $\beta_2 = \beta_1 + 1$.
In particular, if $\nu_{\beta_1} = \nu_{\beta_2}$, then $\beta_1$ and $\beta_2$ are two Pisot numbers of degree $2$.
\end{theorem}

The proof of Theorem \ref{main-result} relies on a thorough analysis of the density function of R\'{e}nyi-Parry measures, which will be addressed in the next section.
Taking Hochman-Shmerkin theorem into consideration, we obtain the following corollary immediately.

\begin{corollary}
  Let $\beta_1>1$ be a Pisot number of degree $\ge 3$ and $\beta_2>1$ with $\beta_1 \nsim \beta_2$.
  Then there are no jointly invariant under $T_{\beta_1},T_{\beta_2}$ and ergodic Borel probability measures with positive entropy under $T_{\beta_2}$.
\end{corollary}

We end this section by using Theorem \ref{main-result} to discuss multiplicative independence when the R\'{e}nyi-Parry measures coincide.

\begin{corollary}
  For two non-integers $\beta_1,\beta_2>1$, if $\nu_{\beta_1} = \nu_{\beta_2}$ and $\log \beta_1 / \log \beta_2 \in \mathbb{Q}$, then we have $$\big\{ \beta_1, \beta_2 \big\}= \bigg\{ \frac{1+\sqrt{5}}{2}, \bigg(\frac{1+\sqrt{5}}{2}\bigg)^2 \bigg\}.$$
\end{corollary}
\begin{proof}
  According to Theorem \ref{main-result}, it follows  that $\beta_1$ is the root of equation $x^2-qx-p=0$, where $p,q\in\mathbb{N}$ with $p\leq q$, and $\beta_2 = \beta_1 + 1$.
  Since $\log \beta_1 / \log \beta_2 \in \mathbb{Q}$, there exist $m,n\in \N$ such that $\beta_1^n = \beta_2^m$. That is, $\beta_1^n - (\beta_1 +1)^m=0$.
  Note that $\beta_1$ is an algebraic integer with the minimal polynomial $x^2-qx-p$.
  Thus we have $x^2-qx-p \mid x^n - (x+1)^m$ in $\mathbb{Z}[x]$.
  By taking $x=0$ and $x=-1$, we obtain $p=1$ and $q=1$, respectively.
  It follows that $\beta_1 = (1+\sqrt{5})/2$ and $\beta_2 = \beta_1 +1 = \beta_1^2$, as desired.
\end{proof}

\section{Proof of Theorem \ref{main-result}}\label{sec:proof}
In this section, we will give a complete proof of Theorem \ref{main-result}.
Recall the formula of density function of R\'{e}nyi-Parry measures from \eqref{eq:nor-density}.
Let
\begin{equation}\label{eq:nonor-den}
{h}_\beta(x) = \sum_{ x < T_\beta^n (1)} \frac{1}{\beta^n},\quad x \in [0,1),
\end{equation}
denote the initial density function without normalization constant.
Then we have $$K_\beta = \int_{0}^1 {h}_\beta(x) \;\mathrm{d} x\quad  \text{and} \quad \widetilde{h}_\beta(x) = \frac{h_\beta(x)}{K_\beta}.$$
We first verify the sufficiency of Theorem \ref{main-result} by direct calculation.

\begin{proposition}\label{prop:suff}
If $\beta_1>1$ is the root of equation $x^2-qx-p=0$, where $p,q\in\mathbb{N}$ with $p\leq q$, and $\beta_2=\beta_1+1$, then the R\'{e}nyi-Parry measures for $T_{\beta_1}, T_{\beta_2}$ coincide.
\end{proposition}
\begin{proof}
Note that $\beta_1^2 - q \beta_1 - p =0$, where $p,q\in\mathbb{N}$ with $p\leq q$.
Then we have $q< \beta_1 < q+1$, $T_{\beta_1} (1) = \beta_1 - q$, and $T_{\beta_1}^n (1) =0$ for all $n \ge 2$.
It follows from \eqref{eq:nonor-den} that $${h}_{\beta_1}(x) = \mathds{1}_{[0,1)}(x) +\frac{1}{\beta_1}\mathds{1}_{[0,\beta_1 - q)}(x)$$
where $\mathds{1}_{A}(x)$ is the characteristic function of set $A$.

Recall that $\beta_2=\beta_1+1$. It follows that $T_{\beta_2}(1)=\beta_1 - q$ and $\beta_2 \cdot T_{\beta_2}(1) = (\beta_1+1)(\beta_1 - q) = \beta_1 - q + p$.
This implies that $T_{\beta_2}^n (1) = \beta_1 -q$ for all $n \ge 1$.
Again by \eqref{eq:nonor-den} we obtain
$${h}_{\beta_2}(x) = \mathds{1}_{[0,1)}(x) + \sum_{n=1}^{\f}\frac{1}{\beta_2^n} \cdot \mathds{1}_{[0,\beta_1 - q)}(x) = \mathds{1}_{[0,1)}(x) + \frac{1}{\beta_2-1} \mathds{1}_{[0,\beta_1 - q)}(x) = {h}_{\beta_1}(x).$$
Thus we are led to the conclusion that $\nu_{\beta_1} = \nu_{\beta_2}$.
\end{proof}

The rest of this section is devoted to prove the necessity of Theorem \ref{main-result}.
Suppose that $\beta_1, \beta_2 > 1$ are two non-integers with $\nu_{\beta_1} = \nu_{\beta_2}$.
Then the density functions $\widetilde{h}_{\beta_1}(x) = \widetilde{h}_{\beta_2}(x)$ for Lebesgue almost everywhere $x \in [0,1)$.
First, we will show that the initial density functions ${h}_{\beta_1}(x) = {h}_{\beta_2}(x)$ for each $x\in [0,1)$.
Write $\mathcal{O}_{\beta}$ for the orbit of $1$ under $\beta$-transformation, i.e.,
\begin{equation*}\label{eq:def-orb-1}
\mathcal{O}_{\beta}:=\big\{ T_{\beta}^n (1): n \geq 1 \big\}.
\end{equation*}
Next, we will prove that $\mathcal{O}_{\beta_1}, \mathcal{O}_{\beta_2}$ must be finite sets, and $\mathcal{O}_{\beta_1} \setminus \{0\} = \mathcal{O}_{\beta_2} \setminus \{0\}$.
Moreover, we will show that $0$ is exactly in one of  $\mathcal{O}_{\beta_1}$ and $\mathcal{O}_{\beta_2}$.
Finally, through a more detailed analysis we can complete the proof of necessity.
Before proceeding with the detailed proof, we give some properties of initial density function.

\begin{proposition}\label{prop:property}
  Let $\beta > 1$ be a non-integer. Then we have

    {\rm(i)} $\displaystyle\lim_{x \to 1^{-}} {h}_{\beta}(x) =1$;

    {\rm(ii)} ${h}_\beta(x)$ is decreasing and right continuous on $[0,1)$;

    {\rm(iii)} ${h}_\beta(x)$ is constant on an open interval $(a,b)\subset(0,1)$ if and only if $(a,b) \cap \mathcal{O}_\beta = \emptyset$.
\end{proposition}
\begin{proof}
  For $x \in [0,1)$, define $N_x := \big\{ n \ge 1: T_\beta^n (1) > x \big\}$.
  Note that $T_\beta^0(1) = 1$. From \eqref{eq:nonor-den}, we have $${h}_\beta(x) = 1 +\sum_{n \in N_x} \frac{1}{\beta^n},\; x\in [0,1).$$

  (i) For any $\varepsilon >0$, there exists $\kappa=\kappa(\varepsilon) \in \N$ such that $$\sum_{n=\kappa+1}^{\f} \frac{1}{\beta^n} < \varepsilon.$$
  Set $\delta = \min\big\{ 1 - T_\beta^n (1): n=1,2,\ldots, \kappa \big\}$. For any $1-\delta < x < 1$, we have $N_x \cap \{1,2,\ldots, \kappa\} = \emptyset$, and hence,
  $$1 \le {h}_\beta(x) = 1+ \sum_{n \in N_x} \frac{1}{\beta^n} \le 1 + \sum_{n=\kappa+1}^{\f} \frac{1}{\beta^n} < 1+ \varepsilon.$$
  Thus we get $$ \lim_{x \to 1^{-}} {h}_{\beta}(x) =1.$$

  (ii) Note that $N_{x_2}\subset N_{x_1}$ for any  $0\leq x_1 < x_2 <1$. It follows that $h_\beta(x)$ is decreasing on $[0,1)$. It remains to show the right continuity.

  Fix $x_0 \in [0,1)$.
  For any $\varepsilon >0$ there exists $\kappa=\kappa(\varepsilon) \in \N$ such that $$\sum_{n=\kappa+1}^{\f} \frac{1}{\beta^n} < \varepsilon.$$
  Put $\delta =\min \big\{ T_\beta^n (1) - x_0 : n \in N_{x_0} \cap \{1,2,\ldots, \kappa\} \big\}$ if $N_{x_0} \cap \{1,2,\ldots, \kappa\} \ne \emptyset$; otherwise, put $\delta = 1- x_0$.
  Then we have $(x_0, x_0 + \delta) \cap \big\{ T_\beta^n (1) : 1 \le n \le \kappa \big\} = \emptyset$.
  As a result, for any $x_0 < x < x_0 + \delta$,
  $$(N_{x_0} \setminus N_x) \cap \{1,2,\ldots \kappa\} = \big\{ 1\le n \le \kappa: x_0 < T_\beta^n (1) \le x  \big\} = \emptyset.$$
  It follows that
  $$0 \le h_\beta(x_0) - h_\beta(x) = \sum_{n \in N_{x_0} \setminus N_x} \frac{1}{\beta^n} \le \sum_{n=\kappa+1}^{\f} \frac{1}{\beta^n} < \varepsilon.$$
  Thus we have established that $$\lim_{x \to x_0^+} {h}_\beta(x) = {h}_\beta(x_0).$$
  Since $x_0$ is arbitrarily fixed, we conclude that ${h}_\beta(x)$ is right continuous on $[0,1)$.

  (iii) For sufficiency, suppose on the contrary that there exist $a < x_1 < x_2 < b$ such that ${h}_\beta(x_1)\neq  {h}_\beta(x_2)$.
  Then we have $$ h_\beta(x_1) - h_\beta(x_2) = \sum_{n \in N_{x_1} \setminus N_{x_2} } \frac{1}{\beta^n} >0,$$
  which implies that $N_{x_1} \setminus N_{x_2} \ne \emptyset$.
  Note that $x_1 < T_\beta^n(1) \le x_2$ for any $n \in N_{x_1} \setminus N_{x_2}$.
  This is contrary with $(a,b) \cap \mathcal{O}_\beta = \emptyset$.
  Thus we have derived that ${h}_\beta(x)$ is constant on $(a,b)$.

  For necessary, suppose that $(a,b) \cap \mathcal{O}_\beta \neq  \emptyset$. Then there exists $\ell \in \N$ satisfying $a < T_\beta^\ell (1) < b$.
  Take any $a < x_1 < T_\beta^\ell (1) < x_2 < b$. Then we have $\ell \in N_{x_1}$ and $\ell \not\in N_{x_2}$.
  It follows that $${h}_\beta(x_1) - {h}_\beta(x_2) =  \sum_{n \in N_{x_1} \setminus N_{x_2}} \frac{1}{\beta^n} \ge \frac{1}{\beta^\ell} >0,$$
  a contradiction.
  Therefore we conclude that $(a,b) \cap \mathcal{O}_\beta = \emptyset$.
  The proof is completed.
\end{proof}

The following lemma is an easy exercise in real analysis, which will be used in our subsequent proof.
The proof is left to the reader.

\begin{lemma}\label{lemma:limit}
  Let $f,g: (0,1) \to \R$ be two functions satisfying $f(x) = g(x)$ for Lebesgue almost everywhere $x\in(0,1)$.
  If the left limits $\displaystyle\lim_{x \to x_0^{-}} f(x)$ and $\displaystyle\lim_{x \to x_0^{-}} g(x)$ exist for some $0< x_0 \le 1$, then we have $$\lim_{x \to x_0^{-}} f(x) = \lim_{x \to x_0^{-}} g(x).$$
  Similarly, if the right limits $\displaystyle\lim_{x \to x_0^{+}} f(x)$ and $\displaystyle\lim_{x \to x_0^{+}} g(x)$ exist for some $0 \le x_0 < 1$, then $$\lim_{x \to x_0^{+}} f(x) = \lim_{x \to x_0^{+}} g(x).$$
\end{lemma}

We now proceed to the concrete proof of necessity of Theorem \ref{main-result}. First, we deal with the initial density function.

\begin{proposition}\label{prop:equal}
  If $\beta_1,\beta_2>1$ are two non-integers with $\nu_{\beta_1}=\nu_{\beta_2}$, then ${h}_{\beta_1}(x) = {h}_{\beta_2}(x)$ for each $x \in [0,1)$.
\end{proposition}
\begin{proof}
As $\nu_{\beta_1}=\nu_{\beta_2}$, the density functions $\widetilde{h}_{\beta_1}(x) = \widetilde{h}_{\beta_2}(x)$ for Lebeague almost everywhere $x\in [0,1)$.
It follows from Proposition \ref{prop:property} (i) that $$\lim_{x \to 1^-} \widetilde{h}_{\beta_i}(x) = \frac{1}{K_{\beta_i}} \cdot  \lim_{x \to 1^-} {h}_{\beta_i}(x)= \frac{1}{K_{\beta_i}} \quad\text{for } i =1,2. $$
By Lemma \ref{lemma:limit}, we obtain $$ \frac{1}{K_{\beta_1}} = \lim_{x \to 1^-} \widetilde{h}_{\beta_1}(x) = \lim_{x \to 1^-} \widetilde{h}_{\beta_2}(x) = \frac{1}{K_{\beta_2}}. $$
That is, $K_{\beta_1} = K_{\beta_2}$.
Consequently, the initial density functions ${h}_{\beta_1}(x) = {h}_{\beta_2}(x)$ for Lebeague almost everywhere $x\in [0,1)$.
Recall from Proposition \ref{prop:property} (ii) that the initial density function is right continuous.
Again by Lemma \ref{lemma:limit}, we conclude that for each $x_0 \in [0,1)$,
$$h_{\beta_1}(x_0) = \lim_{x \to x_0^{+}} h_{\beta_1}(x) = \lim_{x \to x_0^{+}} h_{\beta_2}(x) = h_{\beta_2}(x_0).$$
\end{proof}

The orbit of $1$ under $\beta$-transformation will be addressed in the following two propositions.
\begin{proposition}\label{prop:finite}
Let $\beta_1,\beta_2 > 1$ be two different non-integers.
If $\nu_{\beta_1}=\nu_{\beta_2}$, then the sets $\mathcal{O}_{\beta_1}$ and $\mathcal{O}_{\beta_2}$ are finite, and $\mathcal{O}_{\beta_1} \setminus \{0\} = \mathcal{O}_{\beta_2} \setminus \{0\}$.
\end{proposition}
\begin{proof}
  Suppose first that both $\mathcal{O}_{\beta_1}$ and $\mathcal{O}_{\beta_2}$ are infinite.
  Then $T_{\beta_1}^n (1)>0$ and $T_{\beta_2}^n (1)>0$ for all $n \ge 1$.
  From (\ref{eq:nonor-den}) we obtain that $${h}_{\beta_i}(0) = \sum_{n=0}^{\f} \frac{1}{\beta_i^n} = \frac{\beta_i}{\beta_i -1} \quad \text{for } i=1,2.$$
  By Proposition \ref{prop:equal}, we have ${h}_{\beta_1}(0) = {h}_{\beta_2}(0)$.
  This yields that $\beta_1 = \beta_2$, a contradiction.
  Thus we have derived that at least one of sets $\mathcal{O}_{\beta_1}$ and $\mathcal{O}_{\beta_2}$ is finite.

  Without loss of generality, we may assume that $\mathcal{O}_{\beta_1}$ is finite.
  Then write $\mathcal{O}_{\beta_1} \cup \{0,1\} = \{ x_0, x_1, \ldots, x_{\ell+1} \}$, where $0=x_0< x_1  < \cdots < x_\ell < x_{\ell+1}=1$.
  For each $0 \le k \le \ell$, since $(x_{k}, x_{k+1}) \cap \mathcal{O}_{\beta_1} = \emptyset$, it follows from Proposition \ref{prop:property} (iii) that  ${h}_{\beta_1}(x)$ is constant on $(x_{k}, x_{k+1})$.
  According to Proposition \ref{prop:equal}, ${h}_{\beta_2}(x)$ is also constant on $(x_{k}, x_{k+1})$.
  Again by Proposition \ref{prop:property} (iii), we have that $\mathcal{O}_{\beta_2} \cap (x_{k}, x_{k+1}) = \emptyset$ for each $0\le k \le \ell$.
  Thus we conclude that $\mathcal{O}_{\beta_2} \setminus \{0\} \subset \mathcal{O}_{\beta_1} \setminus \{0\}$.
  In particular, $\mathcal{O}_{\beta_2}$ is a finite set.
  Using the same argument, it's easy to show the inverse conclusion that $\mathcal{O}_{\beta_1} \setminus \{0\} \subset \mathcal{O}_{\beta_2} \setminus \{0\}$.
  Therefore we arrive at the conclusion that $\mathcal{O}_{\beta_1}, \mathcal{O}_{\beta_2}$ are finite sets and $\mathcal{O}_{\beta_1} \setminus \{0\} = \mathcal{O}_{\beta_2} \setminus \{0\}$.
\end{proof}

\begin{proposition}\label{prop:0}
  Let $\beta_1,\beta_2 > 1$ be two different non-integers. If $\nu_{\beta_1}=\nu_{\beta_2}$, then the number $0$ is exactly in one of sets $\mathcal{O}_{\beta_1}$ and $\mathcal{O}_{\beta_2}$.
\end{proposition}
\begin{proof}
  The proof is conducted by contradiction. Suppose first that $0 \not\in \mathcal{O}_{\beta_1} \cup \mathcal{O}_{\beta_2}$. Then $T_{\beta_1}^n (1)>0$ and $T_{\beta_2}^n (1)>0$ for all $n \ge 1$.
  We now proceed as in the proof of Proposition \ref{prop:finite}.
  From (\ref{eq:nonor-den}) we find that $${h}_{\beta_i}(0) = \frac{\beta_i}{\beta_i -1} \quad \text{for } i=1,2.$$
  By Proposition \ref{prop:equal}, we have ${h}_{\beta_1}(0) = {h}_{\beta_2}(0)$.
  This implies that $\beta_1 = \beta_2$, a contradiction.

  Next, suppose that $0 \in \mathcal{O}_{\beta_1} \cap \mathcal{O}_{\beta_2}$.
  Let $n_i = \min\{ n \ge 1: T_{\beta_i}^{n+1}(1) =0 \}$ for $i =1,2$.
  Then we have $$ \mathcal{O}_{\beta_i} \setminus \{0\} = \big\{ T_{\beta_i}(1), T_{\beta_i}^2 (1), \ldots, T_{\beta_i}^{n_i}(1) \big\}\quad \text{and}\quad
  {h}_{\beta_i}(0) = \sum_{n=0}^{n_i} \frac{1}{\beta_i^n}.$$
  From Proposition \ref{prop:finite}, we have $\mathcal{O}_{\beta_1} \setminus \{0\} = \mathcal{O}_{\beta_2} \setminus \{0\}$, which yields that $n_1 = n_2$.
  However, it follows from Proposition \ref{prop:equal} that
  $$\sum_{n=0}^{n_1} \frac{1}{\beta_1^n} = {h}_{\beta_1}(0) = {h}_{\beta_2}(0) = \sum_{n=0}^{n_2} \frac{1}{\beta_2^n}.$$
  This is contrary with $\beta_1 \ne \beta_2$.

  Therefore we conclude that the number $0$ is exactly in one of sets $\mathcal{O}_{\beta_1}$ and $\mathcal{O}_{\beta_2}$.
\end{proof}

With the help of the preceding propositions, our main theorem can be proved.

\begin{proof}[Proof of Theorem \ref{main-result}]
The sufficiency follows from Proposition \ref{prop:suff}.
In the following, we focus on the necessity, and recall our hypothesis that $\beta_1, \beta_2 > 1$ are two different non-integers with $\nu_{\beta_1} = \nu_{\beta_2}$.

  From Proposition \ref{prop:0}, the number $0$ is exactly in one of sets $\mathcal{O}_{\beta_1}$ and $\mathcal{O}_{\beta_2}$.
  Without loss of generality, we may assume that $0\in \mathcal{O}_{\beta_1}$ and $0\notin \mathcal{O}_{\beta_2}$.
  Let $m = \min\{ n \ge 1: T_{\beta_1}^{n+1}(1) =0 \}$. Then we have $T_{\beta_1}^n (1)=0$ for all $n > m$.
  Write $x_k = T_{\beta_1}^k (1)$ for $1 \le k \le m$. Note that $x_i\neq x_j$ for all $1 \le i < j \le m$. Thus we have $\mathcal{O}_{\beta_1} \setminus \{0\} = \{ x_1, x_2, \ldots, x_m \}$.
  It follows from (\ref{eq:nonor-den}) that
  \begin{equation}\label{eq:beta-1}
    {h}_{\beta_1}(x) = \mathds{1}_{[0,1)}(x) + \sum_{k=1}^{m} \frac{1}{\beta_1^k} \cdot \mathds{1}_{[0,x_k)}(x).
  \end{equation}

  Note that $0\notin \mathcal{O}_{\beta_2}$. By Proposition \ref{prop:finite}, we have $\mathcal{O}_{\beta_2}= \mathcal{O}_{\beta_1} \setminus \{0\} = \{ x_1, x_2, \ldots, x_m \}$.
  Write $y_k = T_{\beta_2}^k (1)$ for $1 \le k \le m$.
  Then we obtain that $y_i\neq y_j$ for all $1 \le i < j \le m$, and $$\mathcal{O}_{\beta_2} = \{ y_1, y_2, \ldots, y_m \} = \{ x_1, x_2, \ldots, x_m \}.$$
  This means that there exists a unique $\ell\in\{1,\ldots , m\}$ such that $T_{\beta_2}^{m+1}(1) = y_\ell$.
  Consequently, the sequence $\big\{ T_{\beta_2}^n (1) \big\}_{n=1}^\f$ is eventually periodic with period $y_\ell, y_{\ell+1}, \ldots, y_m$.
  By (\ref{eq:nonor-den}) we get that
  \begin{equation}\label{eq:beta-2}
  \begin{split}
    {h}_{\beta_2}(x)
    & = \mathds{1}_{[0,1)}(x) + \sum_{k=1}^{\ell-1} \frac{1}{\beta_2^k} \cdot \mathds{1}_{[0,y_k)}(x) + \sum_{k=\ell}^{m} \bigg( \frac{1}{\beta_2^{k}} \sum_{n=0}^{\f} \frac{1}{\beta_2^{n(m+1-\ell)}} \bigg) \cdot \mathds{1}_{[0,y_k)}(x) \\
    & = \mathds{1}_{[0,1)}(x) + \sum_{k=1}^{\ell-1} \frac{1}{\beta_2^k} \cdot \mathds{1}_{[0,y_k)}(x) + \frac{\beta_2^{m+1-\ell}}{\beta_2^{m+1-\ell} -1} \sum_{k=\ell}^{m}  \frac{1}{\beta_2^{k}} \cdot \mathds{1}_{[0,y_k)}(x).
  \end{split}
  \end{equation}

  According to Proposition \ref{prop:equal}, we have ${h}_{\beta_1}(x)= {h}_{\beta_2}(x)$ for each $x \in [0,1)$.
  By comparing \eqref{eq:beta-1} with \eqref{eq:beta-2}, we draw the conclusion that the sets of coefficients must be equal. That is,
  \begin{equation}\label{eq:coefficient-equal}
    C = C_1 \cup C_2,
  \end{equation}
  where $$ C=\bigg\{ \frac{1}{\beta_1^k}: 1 \le k \le m \bigg\},$$ and  $$C_1= \bigg\{ \frac{1}{\beta_2^k}: 1 \le k \le \ell-1 \bigg\}, \quad C_2=\bigg\{ \frac{\beta_2^{m+1-\ell}}{\beta_2^{m+1-\ell} -1} \cdot \frac{1}{\beta_2^k}: \ell \le k \le m \bigg\}. $$
  By taking $x=0$ in \eqref{eq:beta-1} and \eqref{eq:beta-2}, we obtain that
  \begin{equation}\label{eq:condition-1}
    \sum_{k=0}^{m} \frac{1}{\beta_1^k} = {h}_{\beta_1}(0)= {h}_{\beta_2}(0) =\frac{\beta_2}{\beta_2 -1}.
  \end{equation}
  It follows that $\beta_1 < \beta_2$.

  We first consider the case $m=1$. Then we have $T_{\beta_1}^2(1)=0$, which means that $\beta_1 \cdot T_{\beta_1}(1)\in \N$.
  Let $q$ be the integer part of $\beta_1$, and $p = \beta_1 \cdot T_{\beta_1}(1)$. Clearly, $p,q \in \N$.
  Since $p = \beta_1 \cdot T_{\beta_1}(1) < \beta_1$, we have $p \le q$.
  Note that $T_{\beta_1}(1) = \beta_1 - q$. We conclude that $\beta_1^2 - q \beta_1 - p =0$.
  Finally, by \eqref{eq:condition-1} we obtain that $\beta_2 = \beta_1 +1$.

  For $m \ge 2$, we will derive a contradiction to complete the proof.
  Note that $\beta_1 < \beta_2$. Then we have $1/\beta_1 \not \in C_1$.
  By \eqref{eq:coefficient-equal}, we find that
  \begin{equation}\label{eq:condition-2}
    \frac{1}{\beta_1} = \max C = \max C_2 = \frac{\beta_2^{m+1-\ell}}{\beta_2^{m+1-\ell} -1} \cdot \frac{1}{\beta_2^\ell}.
  \end{equation}
  By applying the fact that $\beta_1 < \beta_2$ again, we obtain $$\frac{1}{\beta_1^2} > \frac{\beta_2^{m+1-\ell}}{\beta_2^{m+1-\ell} -1} \cdot \frac{1}{\beta_2^{\ell+1}}.$$
  This implies that $1/\beta_1^2 \notin C_2$.
  Also by \eqref{eq:coefficient-equal}, we have that $$\frac{1}{\beta_1^2} = \max C_1 = \frac{1}{\beta_2}. $$
  That is,
  \begin{equation}\label{eq:condition-3}
    \beta_2 = \beta_1^2.
  \end{equation}
  From \eqref{eq:condition-2} and \eqref{eq:condition-3}, the sets $C_1$ and $C_2$ can be rewritten as
  $$C_1= \bigg\{ \frac{1}{\beta_1^{2k}}: 1 \le k \le \ell-1 \bigg\} \quad \text{and} \quad C_2=\bigg\{ \frac{1}{\beta_1^{2k+1}} : 0 \le k \le m-\ell \bigg\}. $$
  By considering the minimal value in $C$, we have $\min C = \min C_1$ or $\min C =\min C_2$, which yields that $$m = 2(\ell-1) \quad \text{or} \quad m = 2\ell-1.$$
  No matter what $m$ is, by \eqref{eq:condition-2} and \eqref{eq:condition-3} we always reach the same conclusion that $$\beta_1^{m+1} - \beta_1 - 1=0.$$
  It follows that $1 < \beta_1 < 2$ and hence, $T_{\beta_1}(1) = \beta_1 -1$.
  Note that $0< \beta_1^m \cdot T_{\beta_1}(1) = \beta^{m+1} - \beta_1^m < \beta_1^{m+1} - \beta_1 =1$.
  Thus we have $T_{\beta_1}^{m+1}(1) = \beta_1^m \cdot T_{\beta_1}(1) \in (0,1)$.
  Recall the fact that $T_{\beta_1}^n (1)=0$ for all $n > m$.
  This leads to a contradiction.
  The proof is completed.
\end{proof}

\section*{Acknowledgement}

The authors would like to thank Professor V. Komornik for his helpful discussions for Corollary 1.3. Zhiqiang Wang is supported by the China Postdoctoral Science Foundation (No. 2024M763857), and by the National Natural Science Foundation of China (No. 12471085).

\bibliographystyle{abbrv}

\begin{thebibliography}{1}

\bibitem{Bertrand-1998}
A.~Bertrand-Mathis.
\newblock Sur les mesures simultan\'ement invariantes pour les transformations
  {$x\to\{\lambda x\}$} et {$x\to\{\beta x\}$}.
\newblock {\em Acta Math. Hungar.}, 78(1-2):71--78, 1998.

\bibitem{Hochman-Shmerkin-2015}
M.~Hochman and P.~Shmerkin.
\newblock Equidistribution from fractal measures.
\newblock {\em Invent. Math.}, 202(1):427--479, 2015.

\bibitem{Hofbauer-1978}
F.~Hofbauer.
\newblock {$\beta $}-shifts have unique maximal measure.
\newblock {\em Monatsh. Math.}, 85(3):189--198, 1978.

\bibitem{Parry-1960}
W.~Parry.
\newblock On the {$\beta $}-expansions of real numbers.
\newblock {\em Acta Math. Acad. Sci. Hungar.}, 11:401--416, 1960.

\bibitem{Renyi-1957}
A.~R\'{e}nyi.
\newblock Representations for real numbers and their ergodic properties.
\newblock {\em Acta Math. Acad. Sci. Hungar.}, 8:477--493, 1957.

\bibitem{Takahashi-1973}
Y.~Takahashi.
\newblock Isomorphisms of {$\beta $}-automorphisms to {M}arkov automorphisms.
\newblock {\em Osaka Math. J.}, 10:175--184, 1973.

\end{thebibliography}

\end{document}